\documentclass{article}
\usepackage[centertags]{amsmath}
\usepackage{amsfonts}
\usepackage{amssymb}
\usepackage{amsthm}
\usepackage{latexsym}
\usepackage{amsmath}
\usepackage{amsfonts}
\usepackage[margin=1in]{geometry}
\usepackage{graphicx}
\usepackage{xcolor}
\usepackage{epstopdf}
\usepackage{hyperref}

\newtheorem{theorem}{Theorem}[section]
\newtheorem{lemma}[theorem]{Lemma}

\title{Increasing stability in acoustic and elastic inverse source problems.}

\author{Mozhgan Nora Entekhabi, Victor Isakov}

\begin{document}

\maketitle
\begin{abstract}

We study increasing stability in the  inverse source problem for the Helmholtz equation and classical Lame system from boundary data at multiple wave numbers. By using the Fourier transform with respect to the wave numbers, explicit bounds for analytic continuation, Huygens' principle  and exact  bounds for initial boundary value problems, increasing (with larger wave numbers intervals) stability estimates are  obtained.  
\end{abstract}

\section{Introduction}\label{se_intro}

We are interested in uniqueness and stability in the inverse source problems for elliptic equations and systems when the source term, supported in a bounded domain $\Omega$, is to be found from the data on $\partial\Omega$. One of important examples is recovery of acoustic sources from boundary measurements of the pressure. This type of inverse source problems is also motivated by  wide applications in antenna synthesis, biomedical imaging  and geophysics, in particular, to tsunami prediction.
From the boundary data for one single linear differential equation or system, it is not possible to find the source uniquely \cite[Ch.4]{I}, but in case of family of equations (like the Helmholtz equation for various wave numbers in $(0,K)$) one can regain uniqueness. Then the crucial issue for applications is the stability of the source recovery. In general, a feature of inverse problems for elliptic equations is a logarithmic type stability estimate which results in  a robust recovery of only few parameters describing the source and hence yields very low resolution numerically.
For the Helmholtz equation we will show  increasing (getting nearly Lipschitz) stability when the Dirichlet data are given on the whole 
boundary and $K$ is getting large. Similar results are obtained for the time periodic solutions of the more complicated dynamical elasticity system.

We will use mostly standard notation. $\|u\|_{(l)}$ is the norm of a function $u$ in the Sobolev space $H^l$. $\Omega$ is a bounded domain in ${\mathbb R}^3$ with connected  ${\mathbb R}^3\setminus \bar\Omega$ and the boundary  $\partial \Omega \in C^2$. $C$ denote generic constants depending only on $\Omega$ and in the case of the elasticity system on the Lame parameters
$\lambda, \mu$ and density
$\rho$ which are assumed to be constant. 

Let $u(x,k)$ solve the scattering problem in $\mathbb{R}^3$ with the Sommerfeld radiation condition
\begin{equation}
\label{H}
 (\Delta+k^2) u  = - f_1 -ikf_0\;\mbox{in}\;\mathbb{R}^3,
\end{equation}
\begin{equation}
 \lim r(\partial_r u - ik u) =0\;\text{as}\;
 r=|x|\rightarrow +\infty,
\label{radiation}
\end{equation}
$f_0, f_1\in L_2(\Omega)$,
$f_0=f_1=0$ on ${\mathbb R}^3\setminus \bar\Omega$.

We are interested in uniqueness and stability of recovery of functions $f_0, f_1$ from the near field
data
\begin{equation}
{\label{D}}
   u= u_0  \text{ on }\; \partial \Omega, \text{ when} \; 0<k<K.
\end{equation}

\begin{theorem}\label{thm1}
Let $1<K$.
There exists
 $C$ such that
 \begin{equation}
\label{stability10}
\| f_0\|^2_{(0)}(\Omega)+\|f_1\|^2_{(-1)}(\Omega)
 \leq
C\left(\varepsilon_0^2 +
\frac{M_1^2}{1+K^{\frac{4}{3}} |E_0|^{\frac{1}{2}}}\right),
\end{equation}
\begin{equation}
\label{stability1}
\| f_0\|^2_{(1)}(\Omega)+\|f_1\|^2_{(0)}(\Omega)
 \leq
C\left(\varepsilon_1^2 +
\frac{M_2^2}{1+K^{\frac{4}{3}} |E_1|^{\frac{1}{2}}}\right)
\end{equation}
for all $ u \in H^2(\Omega)$ solving (\ref{H}), (\ref{radiation}).
Here
$$
\varepsilon_0^2 =
\int_0^K \|u( ,\omega)\|^2_{(0)}(\partial\Omega) d\omega,\;
$$
$$
\varepsilon_1^2 =
\int_0^K\left(\omega^2\|u( ,\omega)\|^2_{(0)}(\partial\Omega)+
\|u( ,\omega)\|^2_{(1)}(\partial\Omega)\right) d\omega,\;
$$
$E_j= - \ln \varepsilon_j, j=0,1$ and $
M_1=\|f_0\|_{(1)}(\Omega)+\|f_1\|_{(0)}(\Omega),\;M_2=\|f_0\|_{(2)}(\Omega)+\|f_1\|_{(1)}(\Omega)$.
 \end{theorem}

Next we consider the inverse scattering source problem for  stationary elastic waves.
Let the displacement field ${\bf u}(x,k)$ solve the elasticity system 
\begin{equation}
\label{Helas}
(\mu\Delta+ (\mu+\lambda) \nabla div +\rho k^2) {\bf u}  = - {\bf f}_1 -ik{\bf f}_0\;\mbox{in}\;\mathbb{R}^3,
\end{equation}
 with the following radiation condition
\begin{equation}
\;
 \lim r(\partial_r {\bf u}(x;p) - ic_p^{-1}k {\bf u}(x;p)) =0,\;
  \lim r(\partial_r {\bf u}(x;s) - ic_s^{-1}k {\bf u}(x;s)) =0\;\text{as}\;
 r=|x|\rightarrow +\infty,\label{radiationelas}
 \end{equation}
 where ${\bf u}={\bf u}( ;p)+
{\bf u}( ;s)$ is the Helmholtz decomposition of ${\bf u}$ into compression/pressure and shear waves, $c_p=(\lambda+2\mu)^{\frac{1}{2}}\rho^{-\frac{1}{2}},
c_s=\mu^{\frac{1}{2}}\rho^{-\frac{1}{2}}$. 
Here $\lambda, \mu$ are (constant) Lame parameters satisfying an ellipticity condition $0<\lambda+\mu, 0<\mu$ and $\rho$ is a constant positive density. We are interested in stability of recovery functions ${\bf f}_0, {\bf f}_1$ from the near field
data
\begin{equation}
\label{Delas}
   {\bf u}= {\bf u}_0\;  \text{ on }\; \partial \Omega, \text{ when} \; 0<k<K.
\end{equation}.

\begin{theorem}\label{thm2}
Let $1<K$. There exists
 $C$ such that
\begin{equation}
\label{stability0elas}
\| {\bf f}_0\|^2_{(0)}(\Omega)+\|{\bf f}_1\|^2_{(-1)}(\Omega)
 \leq
C\left(\varepsilon^2 +
\frac{M_{2e}^2}{1+K^{\frac{4}{3}} |E|^{\frac{1}{2}}}\right),
\end{equation}
\begin{equation}
\label{stabilityelas}
\| {\bf f}_0\|^2_{(1)}(\Omega)+\|{\bf f}_1\|^2_{(0)}(\Omega)
 \leq
C\left(\varepsilon_e^2 +
\frac{M_3^2}{1+K^{\frac{4}{3}} |E_e|^{\frac{1}{2}}}\right)
\end{equation}
for all $ {\bf u} \in H^2(\Omega)$ solving (\ref{Helas}), (\ref{radiationelas}).
Here
$$
\varepsilon^2 =
\int_0^K
\|{\bf u}_0( ,\omega)\|^2_{(0)}(\partial\Omega) d\omega,\;
$$
$$
\varepsilon_e^2 =
\int_0^K\left(
\omega^2\|{\bf u}_0( ,\omega)\|^2_{(0)}(\partial\Omega)+
\|{\bf u}_0( ,\omega)\|^2_{(1)}(\partial\Omega)\right) d\omega,\;
$$
$E= - \ln \varepsilon,
E_e= - \ln \varepsilon_e$,  $M_{2e}=\|{\bf f}_0\|_{(2)}(\Omega)+\|{\bf f}_1\|_{(2)}(\Omega)$ and $M_3=\|{\bf f}_0\|_{(3)}(\Omega)+\|{\bf f}_1\|_{(3)}(\Omega)$.
 \end{theorem}

The stability bounds \eqref{stability10},
\eqref{stability1},
\eqref{stability0elas},
\eqref{stabilityelas}
contain a Lipschitz stable parts $C\epsilon^2$ and conditional logarithmic stable part. This logarithmic part is natural and necessary since we deal with elliptic equations and systems. However with growing $K$ logarithmic part
is disappearing and the bounds become nearly Lipschitz.

For the Helmholtz equation uniqueness and numerical results were obtained in  \cite{EV2009}. First increasing stability results 
\cite{BLT2010} handle more particular case and by quite different (direct spatial Fourier analysis) methods.
In \cite{CIL2016} a temporal Fourier transform, sharp bounds of the analytic continuation to higher wave numbers, and exact observability bounds for associated hyperbolic equations  have been introduced to obtain increasing stability bounds for the three dimensional Helmholtz equation.  In \cite{EI} the methods and results of \cite{CIL2016} are extended to the more complicated case of the two dimensional Helmholtz equation. In these works one is using the complete Cauchy data on $\partial\Omega$ instead of \eqref{D}, \eqref{Delas} which are much more realistic: indeed, the common measuring acoustical devise (microphone) registers only pressure $u$, while in seismic one typically collects displacements ${\bf u}$. In \cite{LY} a spherical $\Omega$ was considered and there is a result on increasing stability from only $u$ on $\partial\Omega$, but the used norm of the data was not the standard norm, but it involved the operator of solution of the exterior Dirichlet problem. In the recent preprint \cite{BLZ2018} the results similar to \cite{LY} are obtained for the elasticity system \eqref{Helas}. Moreover, there is a very strong numerical evidence of
improving resolution of the inverse source problem for larger $K$ given in
\cite{BLT2010}, \cite{CIL2016}, \cite{IL} and other publications. 

In this paper we consider an arbitrary domain $\Omega$ and obtain increasing stability bounds \eqref{stability10}, 
\eqref{stability1}, 
\eqref{stability0elas},
\eqref{stabilityelas} which
contain most natural Sobolev norms of the Dirichlet type data \eqref{D}, \eqref{Delas}. These results are new even for the Helmholtz equation, since we only use the Dirichlet data on $\partial\Omega$ and reduce regularity assumptions in \eqref{stability10}. We also handle the classical linear elasticity system.
As in \cite{CIL2016}, in the current work we 
use the Fourier transform in time to reduce our inverse source problem to identification of the initial data in the hyperbolic initial value problem  by lateral Cauchy data (observability in control theory). We derive our increasing stability estimates by using sharp bounds of analytic continuation of the data from $(0,K)$ onto
$(0,+\infty)$ given in \cite{CIL2016} and then subsequently utilized in \cite{EI}, \cite{LY}, \cite{BLZ2018}. A new idea which considerably simplified the proof and hence the exposition is to use the Huygens's principle and known Sakamoto type energy bounds for the corresponding hyperbolic initial boundary value problem (backward in time) to avoid a need in the complete Cauchy data on $\partial\Omega$ and in a direct use of the exact boundary controllability  results. Of course, the Huygens' principle is valid only in odd spatial dimensions and for special  constant coefficients. This restricts extensions of this approach. However, we can handle two very important in applications cases with a minimal amount of technicalities.

More is known about uniqueness and increasing stability in the Cauchy problem for elliptic equations and for identification of the Schr\"odinger potential. 
Classic Carleman estimates imply  some conditional H\"older type stability estimates for solutions of the elliptic Cauchy problem.  In 1960  F. John
\cite{J}, \cite{I}  
showed that for the continuation of the Helmholtz equation from the unit disk onto any larger disk the stability estimate, which is uniform with respect to the 
wave numbers, is still of logarithmic types. In the papers  \cite{HI}, \cite{I1} it was demonstrated that in a certain sense stability is always improving for larger $k$ under (pseudo) convexity conditions on the geometry of the domain and of the coefficients of the elliptic equation. In \cite{IK} there are first analytic and numerical results on improving stability without (pseudo)convexity conditions.  Finally in  \cite{I3} it was shown that it is true for general second order elliptic equations disregard on any convexity conditions.
 
 Increasing stability for the Schr\"odinger potential from the complete set of the boundary data (the Dirichlet-to Neumann map) was demonstrated in \cite{I2}. An extension of this result for the attenuation and conductivity coefficients is obtained
 in \cite{ILW2016}.

The rest of this paper is organized as follows. In Section \ref{se_acoustic}  we adjust and use the methods of \cite{CIL2016}, in particular bounds of the analytic continuation of the
 needed norms of the Dirichlet data from $(0,K)$ onto a sector of the complex plane  $k=k_1+ik_2$, and use them and sharp bounds in \cite{CIL2016} of the harmonic measure of $(0,K)$ in this sector to derive explicit bounds of the analytic continuation of this norms from $(0,K)$ onto the real axis. We use the Huygens' principle and known
 sharp bounds of solutions of the initial boundary value problems for the wave equation for a short derivation of analogues of  exact observability bounds with reduced data in Lemma 2.4. These results are crucial for the proof of Theorem 1.1. In Section
 \ref{se_elastic} the proofs of the preceding section are extended onto much more complicated case of the classical Lame system of the elasticity theory. While the  radiating fundamental solution is not so simple, it has some features which we exploit by usimng integration by parts to adjust the proof of Lemma 2.1 to get needed bounds of the analytic continuation in Lemma 3.1. In Lemma 3.2 we obtain some rates of decay of scattering solutions of the elasticity system which are less precise than the corresponding result of Lemma 2.2 because  we are not aware of sharp regularity results available  scalar transmission problems.  The Huygens' principle is also valid for the initial  value problem for the associated dynamical elasticity system. By using it and  the  versions \cite{BL2002}, \cite{GZ} of sharp boundary regularity results of \cite{LLT86}, \cite{S} for the initial boundary value  dynamical elastic we obtain crucial bounds of the initial data by the lateral displacements. After that the proof for elasticity system proceeds as in section 2 for the Helmholtz equation.

\section{Increasing stability for acoustic waves}\label{se_acoustic}

 The well-known integral representation for (\ref{H}), (\ref{radiation}) yields
\begin{equation}
\label{int}
u(x,k) =  \frac{1}{4\pi}\int_{\Omega}
(f_1(y)+ikf_0(y))\frac{e^{ik|x-y|}}{|x-y|} dy.
\end{equation}

Due to (\ref{int}),
\begin{equation}
\label{int0}
\int_{-\infty}^{\infty}\|u(,\omega)\|_{(0)}^2(\partial\Omega)d\omega  = I_0(k) + \int_{k< |\omega|} \|u(,\omega)\|_{(0)}^2(\partial\Omega)d\omega, 
\end{equation}
\begin{equation}
\label{int1}
\int_{-\infty}^{\infty} \omega^2\|u(,\omega)\|_{(0)}^2(\partial\Omega)d\omega = I_1(k) + \int_{k< |\omega|} \omega^2\|u(,\omega)\|_{(0)}^2(\partial\Omega)d\omega,
\end{equation}
\begin{equation}
\label{int2}
\int_{-\infty}^{\infty} \|\nabla_{\tau} u(,\omega)\|_{(0)}^2 (\partial\Omega)d\omega  = I_2(k) + \int_{k< |\omega|} \|\nabla_{\tau} u(,\omega)\|_{(0)}^2 (\partial\Omega)d\omega,
\end{equation}
where $I_0(k), I_1(k)$ and $I_2(k)$ are defined as
\begin{equation}
I_0(k) = 2\int_0^k \int_{\partial\Omega}
\left(\int (f_1(y)+i\omega f_0(y))
\frac{e^{i\omega|x-y|}}{|x-y|} dy\right)
\left(\int (f_1(y)-i\omega f_0(y))
\frac{e^{-i\omega|x-y|}}{|x-y|} dy\right) d \Gamma(x) d\omega, \label{I0}
\end{equation}
\begin{equation}
I_1(k) = 2\int_0^k \omega^2\int_{\partial\Omega}
\left(\int (f_1(y)+i\omega f_0(y))
\frac{e^{i\omega|x-y|}}{|x-y|} dy\right)
\left(\int (f_1(y)-i\omega f_0(y))
\frac{e^{-i\omega|x-y|}}{|x-y|} dy\right) d \Gamma(x) d\omega, \label{I1}
\end{equation}
\begin{equation}
I_2(k) = 2\int_0^k\int_{\partial\Omega}
\left(\int (f_1(y)+i\omega f_0(y))
\nabla_{\tau,x}\frac{e^{i\omega|x-y|}}{|x-y|} dy \right)
\left(\int (f_1(y)-i\omega f_0(y))
\nabla_{\tau,x}\frac{e^{-i\omega|x-y|}}{|x-y|} dy\right)
d \Gamma(x) d\omega , \label{I2}
\end{equation}
$\nabla_{\tau}$ is the tangential projection of the gradient and we used that $\overline{u(x,\omega)}= u (x,-\omega)$.
Due to the definitions of the following norms of the boundary data:
\begin{equation}
\label{varepsilon}
\varepsilon_0^2 = I_0(K),\;
\varepsilon_1^2 = I_1(K)+I_2(K).
\end{equation}
 The truncation level $k$ in (\ref{I0}),(\ref{I1}) and (\ref{I2}) is important to keep balance between the known data and the unknown information when $k\in [K,\infty)$.

Since the integrands are entire analytic functions of $\omega$, the integrals in (\ref{I0}),(\ref{I1}), (\ref{I2}) with respect to $\omega$ can be taken over any path joining points $0$ and $k$ of the complex plane.
Thus $I_1(k), I_2(k)$ are entire analytic functions of $k=k_1+ik_2$. As in \cite{CIL2016}, we will need the following elementary estimates.

\begin{lemma}
Let $supp f_0, supp f_1 \subset \Omega$ and $f_1 \in H^1(\Omega)$, $f_0\in H^1(\Omega)$. Then
\begin{equation}
\label{boundI0}
|I_0(k)|\leq
8\pi|\partial\Omega| D \left(|k|
\|f_1\|^2_{(0)}(\Omega)+
\frac{1}{3}|k|^3
\|f_0\|^2_{(0)}(\Omega)\right)e^{2 D|k_2|},
\end{equation}
\begin{equation}
\label{boundI1}
|I_1(k)|\leq
8\pi|\partial\Omega| D \left(\frac{1}{3}|k|^3
\|f_1\|^2_{(0)}(\Omega)+
\frac{1}{5}|k|^5
\|f_0\|^2_{(0)}(\Omega)\right)e^{2 D|k_2|},
\end{equation}
\begin{equation}
\label{boundI2}
|I_2(k)|\leq
8\pi|\partial\Omega| D \left(|k|
\|f_1\|^2_{(1)}(\Omega)+
{\frac{1}{3}|k|^3
\|f_0\|^2_{(1)}(\Omega)}\right)e^{2 D|k_2|},
\end{equation}
where $|\partial \Omega|$ is the area of $\partial \Omega$ and $ D= \sup |x-y|$ over
$x,y \in \Omega$.
\end{lemma}

\begin{proof}
Using the parametrization $\omega = ks, s \in (0,1)$ in the line integral and the elementary inequality
$|e^{i\omega|x-y|}|\leq e^{|k_2|D }$
it is easy to derive that
\begin{align*}
|I_0(k)| & \leq
2\int_0^1|k| 
\left(\int_{\partial\Omega}
\left(\int_{\Omega} (|f_1(y)|+|k|s |f_0(y)|)
\frac{e^{|k_2| D}}{|x-y|} dy\right)^2
 d \Gamma(x) \right) d s \\
 & \leq
4\int_0^1|k|
\int_{\partial\Omega}
\left(\int_{\Omega} (|f_1(y)|^2+|k|^2s^2 |f_0(y)|^2)\right)dy
\left(\int_{\Omega} \frac{e^{2|k_2|D}}{|x-y|^2} dy\right)
 d \Gamma(x)  d s,
\end{align*}
where the Schwarz inequality is used for the integrals with respect to $y$. The polar coordinates $r=|y-x|$ (originated at $x$) with respect to $y$ yield
$$
|I_0(k)| \leq 4|k|
\int_0^1
 \left(\int_{\Omega}\left( 
 |f_1|^2(y)
+|k|^2s^2|f_0|^2 \right)dy\right) ds
\int_{\partial\Omega} \left(4\pi\int_0^D e^{2|k_2|D} d r \right) d \Gamma(x).
$$
Integrating with respect to $s$, $x$, and $r$, we complete the proof of (\ref{boundI0}).

The bound (\ref{boundI1}) is derived in \cite{CIL2016}, Lemma 2.1,   in a similar way.

The bound (\ref{boundI2}) is also actually obtained in \cite{CIL2016}.

Indeed,
$$
|I_2(k)| \leq 2|k|
\int_0^1\int_{\partial\Omega}
 |\nabla_{\tau}
 \int_{\Omega}\left( 
 f_1(y)
+ksf_0(y)\right)\frac{e^{iks|x-y|}}{|x-y|} dy |^2 d\Gamma(x) ds\leq
$$
$$
2|k|\int_0^1\int_{\partial\Omega}
 |\nabla_{x}
 \int_{\Omega}\left( 
 f_1(y)
+ksf_0(y)\frac{e^{iks|x-y|}}{|x-y|} \right)dy |^2 d\Gamma(x) ds.
$$
Observing that
$\nabla_x\frac{e^{iks|x-y|}}{|x-y|}=-\nabla_y\frac{e^{iks|x-y|}}{|x-y|}$ and integrating by parts with respect to $y$
we reduce the proof to the above proof of \eqref{boundI0}.

\end{proof}

The following steps are needed to link the unknown values of $I_0(k), I_1(k)$ and $I_2(k)$ for $k\in [K,\infty)$  to the known values $\varepsilon_j$ defined in (\ref{varepsilon}).
 
Let $S$ be the sector $\{k: -\frac{\pi}{4} < \arg k < \frac{\pi}{4} \}$ and $\mu(k)$
be the harmonic measure of the interval
$[0,K]$ in $S\setminus [0,K]$. Observe that $|k_2|\leq k_1$ when $k\in S$, so
$$
|I_0(k)e^{-2(D+1)k}|\leq
C \left(|k_1|
\|f_1\|^2_{(0)}(\Omega)+
|k_1|^3
\|f_0\|^2_{(0)}(\Omega)\right)e^{-2k_1} \leq C M_0^2,
$$
with $M_0= \|f_0\|_{(0)}(\Omega) +\|f_1\|_{(0)}(\Omega)$ and  generic constants $C$.
Noticing that
$$
|I_0(k)e^{-2(D+1)k}|\leq \varepsilon_0^2\;
\text{on}\;[0,K],
$$
we conclude that
\begin{equation}
\label{I0omega}
|I_0(k)e^{-2(D+1)k}|\leq
C \varepsilon_0^{2\mu(k)}M_0^2,
\end{equation}
when $K<k<+\infty$. 

Similar arguments also yield
\begin{equation}
\label{I1omega}
|I_1(k)e^{-2(D+1)k}|\leq
C \varepsilon^{2\mu(k)}M_0^2,
\end{equation}
\begin{equation}
\label{I2omega}
|I_2(k)e^{-2(D+1)k}|\leq
C \varepsilon^{2\mu(k)}M_1^2.
\end{equation}

We need  the following lower bound of the harmonic measure $\mu(k)$ given in \cite{CIL2016}, Lemma 2.2.

\begin{lemma}
If $0<k< 2^{\frac{1}{4}}K$, then
\begin{equation}
\label{boundharm1}
\frac{1}{2}\leq \mu(k).
\end{equation}

If $ 2^{\frac{1}{4}}K < k$, then
\begin{equation}
\label{boundharm2}
\frac{1}{\pi}\left(\left(\frac{k}{K}\right)^4-1\right)^{-\frac{1}{2}} \leq \mu(k).
\end{equation}
\end{lemma}

We consider the hyperbolic initial value problem
\begin{equation}
\label{Cauchy}
\partial_t^2 U -\Delta U=0 \;\text{on}\;\mathbb{R}^3\times(0,+\infty),\;
U( ,0)=f_0,\;\partial_t U( ,0)=-f_1\;
\text{on}\;\mathbb{R}^3.
\end{equation}
We define $U(x,t)=0$ when $t<0$.
As shown in \cite{CIL2016}, section 4, the solution of \eqref{H} coincides with the Fourier transform  of $U$, namely
\begin{equation}
\label{uk}
u(x,k)=
\frac{1}{\sqrt{2\pi}}\int_{-\infty}^{+\infty}
U(x,t) e^{ikt} dt.
\end{equation}

To proceed, the estimate for remainders in \eqref{int0}, \eqref{int1}, \eqref{int2} is used and summarized 
in the next result proven as Lemma 4.1 in \cite{CIL2016}.

\begin{lemma}
\label{lemma_apriori}
Let $u$ be a solution to the forward problem
(\ref{H}), (\ref{radiation}) with $f_1 \in H^1(\Omega)$ and $f_0\in H^2(\Omega)$, $ supp f_0, supp f_1 \subset \Omega$. Then
\begin{equation}
\label{apriori0}
\int_{k<|\omega|}
\|u( ,\omega)\|^2_{(0)}(\partial\Omega) d\omega \leq C k^{-2}\left(\|f_0\|^2_{(1)}(\Omega)+
\|f_1\|^2_{(0)}(\Omega)\right),
\end{equation}
\begin{equation}
\label{apriori}
\int_{k<|\omega|}
\omega^2\|u( ,\omega)\|^2_{(0)}(\partial\Omega) d\omega +
\int_{k<|\omega|}
\|\nabla u( ,\omega)\|^2_{(0)}(\partial\Omega) \leq C k^{-2}\left(\|f_0\|^2_{(2)}(\Omega)+
\|f_1\|^2_{(1)}(\Omega)\right).
\end{equation}
\end{lemma}

\begin{proof}

We claim that
\begin{equation}
\label{energy}
\| U\|^2_{(1)}
(\partial\Omega \times (0,D)) \leq
C\left(\|f_0\|^2_{(1)}(\Omega)+
\|f_1\|^2_{(0)}(\Omega)\right).
\end{equation}
Indeed, the initial value problem (\ref{Cauchy})
can be viewed as the (strictly) hyperbolic
transmission problem
$$
\partial_t^2U^--\Delta U^-=0\;\text{in}\;\Omega\times(0,+\infty),\;
\partial_t^2U^+-\Delta U^+=0\;\text{in}\;(\mathbb R^3 \setminus \bar\Omega)\times(0,+\infty),\;
$$
$$
U^-= -f_0,\;\partial_tU^-= f_1\;\text{in}\;\Omega\times\{0\},\;
U^+=-f_0,\;\partial_tU^+=f_1\;\text{in}\;
(\mathbb R^3 \setminus \bar\Omega)\times\{0\},
$$
$$
U^+-U^-= g_0,\;
\partial_{\nu}U^+-
\partial_{\nu}U^-
= g_1\;\text{on}\;\partial\Omega\times (0,+\infty),
$$
with $g_0=0, g_1=0$. Then (\ref{energy}) follows
from the generalization \cite{I1} of Sakamoto
results \cite{S} onto transmission problems.

Obviously, the following inequalities hold
\begin{align*}
\int_{k<|\omega|}
\|u( ,\omega)\|^2_{(0)}(\partial\Omega)
d\omega
& \leq k^{-2}
\int_{k<|\omega|}
\omega^2\|u( ,\omega)\|^2_{(0)}(\partial\Omega)
d\omega \\
& \leq
k^{-2} \int_{\mathbb R}
\omega^2\|u( ,\omega)\|^2_{(0)}(\partial\Omega)
d\omega =  k^{-2}
\int_{\mathbb R}
\|\partial_t U( ,t)\|^2_{(0)}(\partial\Omega)
dt
\end{align*}
by the Parseval's identity. Since according to the Huygens' principle $U( ,t)=0$ when $D<t$,
the last inequality combined with (\ref{energy})
implies the bound (\ref{apriori0}).

(\ref{apriori}) is similarly  derived in \cite{CIL2016}, Lemma 4.1.

\end{proof}

The main novelty compared with \cite{CIL2016} is the next result which follows almost immediately from the Huygens' principle for the initial value problem and the known bounds for  initial boundary value hyperbolic problems.

\begin{lemma}
\label{lem_observ}
Let $U$ be a solution to
(\ref{Cauchy}) with $f_1 \in L^2(\Omega)$,$f_0\in H^1(\Omega)$, $ supp f_0, supp f_1\subset \Omega$.
Then
\begin{equation}
\label{exact0}
\|f_0\|^2_{(0)}(\Omega)+
\|f_1\|^2_{(-1)}(\Omega) \leq
C\| U\|^2_{(0)}(\partial\Omega\times (0, D)),
\end{equation}
\begin{equation}
\label{exact}
\|f_0\|^2_{(1)}(\Omega)+
\|f_1\|^2_{(0)}(\Omega) \leq
C\left(\|\partial_t U\|^2_{(0)}(\partial\Omega\times (0, D))+
\| U \|^2_{(1)}(\partial\Omega\times(0,D))\right).
\end{equation}
\end{lemma}

\begin{proof}

Since $supp f_0, supp f_1 \subset \Omega$ from the Huygens' principle it follows that $U=0$ on $\Omega\times (D,+\infty)$. 
Now \eqref{exact0}, \eqref{exact} follow from the
generalizations \cite{LLT86}, Theorem 4.1, of Sakamoto energy estimates \cite{S} for the initial boundary value problem applied to 
the initial boundary value problem
$$
\partial_t^2 U -\Delta U=0 \;\text{on}\;\mathbb{R}^3\times(0,D),\;
U( ,D)=0,\;\partial_t U( ,D)=0\;
\text{on}\;\Omega.
$$

\end{proof}

Finally, we are ready to prove the increasing stability estimate of Theorem 1.1.

\begin{proof}
Without loss of generality, we can assume that $\varepsilon_0 <1$ and
$\pi(D+1)E_0^{-\frac{1}{4}} <1$, otherwise the bound
(\ref{stability10}) is straightforward.

In \eqref{int0},\eqref{int1},\eqref{int1} we let
\begin{equation}
\label{K}
k=K^{\frac{2}{3}} E_0^{\frac{1}{4}},\;
\text{when}\; 2^{\frac{1}{4}}K^{\frac{1}{3}}<
E_0^{\frac{1}{4}},\;\text{and}\;k=K,\;
\text{when}\;E_0^{\frac{1}{4}}\leq
2^{\frac{1}{4}}K^{\frac{1}{3}}.
\end{equation}
If $2^{\frac{1}{4}}K^{\frac{1}{3}}<
E_0^{\frac{1}{4}}$, then from (\ref{boundharm2}),
(\ref{I1omega}), and (\ref{K}) we obtain
\begin{align*}
|I_0(k)| & \leq
e^{2(D+1)k} e^{-\frac{2}{\pi}
\left(\left(\frac{k}{K}\right)^4 -1\right)^{-\frac{1}{2}}E_0} CM_0^2 \\
&\leq
CM_0^2 e^{2(D+1)K^{\frac{2}{3}}E_0^{\frac{1}{4}}-
 \frac{2}{\pi}\left(\frac{K}{k}\right)^2 E_0} =
 CM_0^2 e^{-2K^{\frac{2}{3}} \frac{1}{\pi}E_0^{\frac{1}{2}}\left(1-\pi(D+1)E_0^{-\frac{1}{4}}\right)}.
\end{align*}
Using the assumption at the beginning of the proof and the elementary inequality $e^{-y}\leq \frac{6}{y^3}$ 
 when $0<y$, we conclude that
\begin{equation}
\label{I0h}
|I_0(k)|\leq
CM_0^2 \frac{1}{K^2 E_0^{\frac{3}{2}}\left(1-\pi(D+1)E_0^{-\frac{1}{4}}\right)^3}.
\end{equation}
On the other hand, if $E_0^{\frac{1}{4}} \leq 2^{\frac{1}{4}}K^{\frac{1}{3}}$, then $k=K$ and  from (\ref{K}) we derive that
\begin{equation}
\label{I0l}
|I_0(k)|\leq 2\varepsilon_0^2.
\end{equation}

Hence, using  (\ref{int0}), (\ref{K}), (\ref{I0h}), and
 (\ref{I0l}) we yield
\begin{align}
\int_{\partial\Omega}\int_{-\infty}^{+\infty}
 |u(x,\omega)|^2 d\omega d\Gamma(x)
& =
I_0(k)+ \int_{\partial\Omega}\int_{k<|\omega|}
 |u((x,\omega)|^2 d\omega d\Gamma(x) \nonumber \\
& \leq \varepsilon_0^2+ CM_0^2\frac{1}{K^2 E_0^{\frac{3}{2}}}+
C\frac{\|f_0\|_{(2)}^2+\|f_1\|_{(1)}^2}
{1+K^{\frac{4}{3}}E_0^{\frac{1}{2}}} , \label{data1}
\end{align}
where we used also (\ref{apriori0}).

By Lemma \ref{lem_observ}, we finally derive
$$
\|f_0\|^2_{(0)}(\Omega)+
\|f_1\|^2_{(-1)}(\Omega)  \leq
C\|U\|^2_{(0)}
(\partial\Omega\times (0, D))   \leq
C\| U\|^2_{(0)}(\partial\Omega\times {\mathbb R}) \leq
$$
$$
C \int_{\partial\Omega}
\int_{-\infty}^{+\infty}
|u(x,\omega)|^2 d\omega d\Gamma(x) \leq
C \left(\varepsilon_0^2+ M_1^2\frac{1}{K^2 E_0^{\frac{3}{2}}}+
\frac{\|f_0\|_{(1)}^2+\|f_1\|_{(0)}^2}
{1+K^{\frac{4}{3}}E_0^{\frac{1}{2}}} \right)
$$
due to the Parseval's identity and (\ref{data1}). Since
$$
K^{\frac{4}{3}}E_0^{\frac{1}{2}}< K^2  E_0^{\frac{3}{2}},
$$
when $1<K, 1<E_0$, the proof  of \eqref{stability10} is complete.

 \eqref{stability1} can be proved similarly when we use
 \eqref{exact} instead of \eqref{exact0}

\end{proof}

\section{Increasing stability for elastic  waves}
\label{se_elastic}

 The well-known integral representation for \eqref{Helas},\eqref{radiationelas} yields
\begin{equation}
\label{intelas}
{\bf u}(x,k) =  \int_{\Omega}
{\bf \Phi}(x-y;k)({\bf f}_1(y)+ik{\bf f}_0(y)) dy,
\end{equation}
where
$$
{\bf \Phi}(x-y;k)=\frac{e^{ic_s^{-1}k|x-y|}}{4\pi c_s^2|x-y|}{\bf I}_3+ k^{-2}\nabla_x^2
\frac{e^{ic_s^{-1}k|x-y|}-
e^{ic_p^{-1}k|x-y|}}{4\pi |x-y|},
$$
${\bf I}_3$ is the $3\times 3$ identity matrix and $\nabla^2 \Phi$ is the $3\times 3$ matrix $(\partial_j\partial_m \Phi), j,m=1,2,3$. Observe that
\begin{equation}
\label{Phi}
{\bf \Phi}(x-y;k)=\frac{e^{ic_s^{-1}k|x-y|}}{4\pi c_s^2|x-y|}{\bf I}_3+ k^{-2}\nabla_x^2
\frac{e^{ic_s^{-1}k|x-y|}-
e^{ic_p^{-1}k|x-y|}-ik(c_s^{-1}-c_p^{-1})|x-y|}{4\pi |x-y|}.
\end{equation}

Due to (\ref{intelas}),
\begin{equation}
\label{int0elas}
\int_{-\infty}^{\infty} \|{\bf u}(,\omega)\|_{(0)}^2(\partial\Omega)d\omega  = I_{0,e}(k) + \int_{k< |\omega|} \|{\bf u}(,\omega)\|_{(0)}^2(\partial\Omega)d\omega,
\end{equation} 
\begin{equation}
 \label{int1elas}
\int_{-\infty}^{\infty} \omega^2\|{\bf u}(,\omega)\|_{(0)}^2(\partial\Omega)d\omega = I_{1,e}(k) + \int_{k< |\omega|} \omega^2\|{\bf u}(,\omega)\|_{(0)}^2(\partial\Omega)d\omega, 
\end{equation}
\begin{equation}
\label{int2elas}
\int_{-\infty}^{\infty} \|\nabla_{\tau} {\bf u}(,\omega)\|_{(0)}^2 (\partial\Omega)d\omega = I_{2,e}(k) + \int_{k< |\omega|} \|\nabla_{\tau} {\bf u}(,\omega)\|_{(0)}^2 (\partial\Omega)d\omega,
\end{equation}
where $I_{1,e}(k)$ and $I_{2,e}(k)$ are defined as
\begin{equation*}
I_{0,e}(k) =
\end{equation*}
\begin{equation}
 2\int_0^k \int_{\partial\Omega}
\left(\int_{\Omega} {\bf \Phi}(x-y;\omega)({\bf f}_1(y)+i\omega {\bf f}_0(y)) dy\right)
\left(\int_{\Omega} {\bf \Phi}(x-y;-\omega)({\bf f}_1(y)-i\omega {\bf f}_0(y))
 dy\right) d \Gamma(x) d\omega, \label{I0e}
\end{equation}
$$
I_{1,e}(k) =
$$
\begin{equation}
 2\int_0^k \omega^2\int_{\partial\Omega}
\left(\int_{\Omega} {\bf \Phi}(x-y;\omega)({\bf f}_1(y)+i\omega {\bf f}_0(y)) dy\right)
\left(\int_{\Omega} {\bf \Phi}(x-y;-\omega)({\bf f}_1(y)-i\omega {\bf f}_0(y))
 dy\right) d \Gamma(x) d\omega, \label{I1e}
\end{equation}
$$
I_{2,e}(k) =
$$
\begin{equation}
 2\int_0^k\int_{\partial\Omega}
\left(\int_{\Omega} {\bf \nabla_{\tau,x}\Phi}(x-y;\omega)({\bf f}_1(y)+i\omega {\bf f}_0(y)) dy \right)
\left(\int_{\Omega} \nabla_{\tau,x}{\bf \Phi}(x-y;-\omega)({\bf f}_1(y)-i\omega {\bf f}_0(y))
 dy\right)
d \Gamma(x) d\omega , \label{I2e}
\end{equation}
and we used that $\overline{{\bf u}(x,\omega)}= {\bf u} (x,-\omega)$.

Using \eqref{Phi} and the power series for the exponential function we can see that ${\bf \Phi}$ is an entire analytic function of $k=k_1+ik_2$, so, as in section \ref{se_acoustic}, $I_{0,e}(k), I_{1,e}(k), I_{2,e}(k)$ are entire analytic functions of $k$. 

\begin{lemma}
Let $supp {\bf f}_0, supp {\bf f}_1 \subset \Omega$ and ${\bf f_1} \in H^3(\Omega)$, ${\bf f}_0\in H^3(\Omega)$. Then
\begin{equation}
\label{boundI0elas}
|I_{0,e}(k)|\leq
C \left(|k|
\|{\bf f}_1\|^2_{(2)}(\Omega)+
|k|^3
\|{\bf f}_0\|^2_{(2)}(\Omega)\right)e^{2 Dc_s^{-1}|k_2|},
\end{equation}
\begin{equation}
\label{boundI1elas}
|I_{1,e}(k)|\leq
C \left(|k|^3
\|{\bf f}_1\|^2_{(2)}(\Omega)+
|k|^5
\|{\bf f}_0\|^2_{(2)}(\Omega)\right)e^{2 Dc_s^{-1}|k_2|},
\end{equation}
\begin{equation}
\label{boundI2elas}
|I_{2,e}(k)|\leq
C \left(|k|
\|{\bf f}_1\|^2_{(3)}(\Omega)+
|k|^3
\|{\bf f}_0\|^2_{(3)}(\Omega)\right)e^{2 D c_s^{-1}|k_2|},
\end{equation}
where $ D= \sup |x-y|$ over
$x,y \in \Omega$.
\end{lemma}

\begin{proof}
 It is easy to see that
$$
\nabla^2_x
\frac{e^{ic_s^{-1}k|x-y|}-
e^{ic_p^{-1}k|x-y|}-ik(c_s^{-1}-c_p^{-1})|x-y|}{4\pi |x-y|} = \nabla^2_y
\frac{e^{ic_s^{-1}k|x-y|}-
e^{ic_p^{-1}k|x-y|}-ik(c_s^{-1}-c_p^{-1})|x-y|}{4\pi |x-y|} .
$$
Therefore, integrating by parts with respect to $y$ in \eqref{intelas} we yield
$$
{\bf u}(x,k) =  \int_{\Omega}
(\frac{e^{ic^{-1}k|x-y|}}{4\pi c^2|x-y|}({\bf f}_1(y)+ik{\bf f}_0(y))-
$$
\begin{equation}
\label{intelas0}
k^{-2}
\frac{e^{ic_s^{-1}k|x-y|}-
e^{ic_p^{-1}k|x-y|}-ik(c_s^{-1}-c_p^{-1})|x-y|}{4\pi |x-y|}(\nabla div {\bf f}_1(y)+
ik\nabla div {\bf f}_0(y))) dy.
\end{equation}
Observe that
\begin{equation}
\label{boundPhi}
|k^{-2}
\frac{e^{ic_s^{-1}k|x-y|}-
e^{ic_p^{-1}k|x-y|}-ik(c_s^{-1}-c_p^{-1})|x-y|}{4\pi |x-y|}|\leq C
\frac{e^{c_s^{-1}|k_2||x-y|}}{|x-y|},
\end{equation}
or
$$
|e^{ic_s^{-1}k|x-y|}-
e^{ic_p^{-1}k|x-y|}-ik(c_s^{-1}-c_p^{-1})|x-y|| \leq C
|k|^2e^{c_s^{-1}|k_2||x-y|}{|x-y|}.
$$
Indeed, when $1\leq|k|$ the later bound is obvious, since $c_s<c_p$. When $|k|<1$, then using the series for the 
exponential function we conclude that the left hand side is
$$
|\sum_{m=2}^{+\infty}
((c_s^{-m}-c_p^{-m})
\frac{(ik|x-y|)^m}{m!}|\leq
|k|^2c_s^{-2}|x-y|^2
\sum_{m=2}^{+\infty}
c_s^{-(m-2)}
\frac{(k|x-y|)^{m-2}}{(m-2)!}\leq C|k|^2.
$$
Since $C\leq e^{c_s^{-1}|k_2||x-y|}$, it completes the derivation of \eqref{boundPhi}.
Using \eqref{intelas0} and \eqref{boundPhi} the proof of
\eqref{boundI0elas} proceeds exactly as in the proof of Lemma 2.1.

To demonstrate \eqref{boundI1elas} by using
\eqref{I1e}, \eqref{Phi},
\eqref{intelas}, \eqref{boundPhi} as in the proof of Lemma 2.1 we yield
$$
|I_{1,e}(k)| \leq C
\int_0^1 |k|^3s^2
 \left(\int_{\Omega}\left( 
|{\bf f}_1|^2(y)
+|k|^2s^2|{\bf f}_0|^2+
|\nabla div {\bf f}_1|^2(y)
+|k|^2s^2|\nabla div{\bf f}_0|^2 \right)dy\right)ds 
$$
$$
\int_{\partial\Omega} \left(\int_0^D e^{2c_s^{-1}|k_2|D} d r \right) d \Gamma(x) .
$$
Integrating with respect to $s, r$ we obtain
\eqref{boundI1elas}.

\eqref{boundI2elas} can be derived similarly to
\eqref{boundI0elas} as in Lemma 2.1, integrating by parts and using \eqref{intelas0} and \eqref{boundPhi}.

\end{proof}

Observe that $|k_2|\leq k_1$ when $k\in S$, so
$$
|I_{0,e}(
k)e^{-2(D+1)k}|\leq
C \left(|k_1|^3
\|{\bf f}_1\|^2_{(2)}(\Omega)+
|k_1|^5
\|{\bf f}_0\|^2_{(2)}(\Omega)\right)e^{-2k_1} \leq C M_{2,e}^2.
$$
Noticing that
$$
|I_{0,e}(k)e^{-2(D+1)k}|\leq \varepsilon^2\;
\text{on}\;[0,K],
$$
we conclude that
\begin{equation}
\label{I0omegaelas}
|I_{0,e}(k)e^{-2(D+1)k}|\leq
C \varepsilon^{2\mu(k)}M_{2,e}^2,
\end{equation}
when $K<k<+\infty$. Similar arguments also yield
\begin{equation}
\label{I1omegaelas}
|I_{1,e}(k)e^{-2(D+1)k}|\leq
C \varepsilon^{2\mu(k)}M_3^2,
\end{equation}
\begin{equation}
\label{I2omegaelas}
|I_{2,e}(k)e^{-2(D+1)k}|\leq
C \varepsilon^{2\mu(k)}M_3^2,
\end{equation}
when $K<k<+\infty$.

We consider the dynamical initial value problem
\begin{equation}
\label{Cauchyelas}
\rho\partial_t^2 {\bf U} -\mu\Delta {\bf U}-(\lambda+\mu)\nabla div {\bf U}=0 \;\text{on}\;\mathbb{R}^3\times(0,+\infty),\;
{\bf U}( ,0)={\bf f}_0,\;\partial_t {\bf U}( ,0)=-{\bf f}_1\;
\text{on}\;\mathbb{R}^3.
\end{equation}
We define ${\bf U}(x,t)=0$ when $t<0$.
As above we claim that the solution of \eqref{Helas} coincides with the Fourier transform  of ${\bf U}$, namely
\begin{equation}
\label{ukelas}
{\bf u}(x,k)=
\frac{1}{\sqrt{2\pi}}\int_{-\infty}^{+\infty}
{\bf U}(x,t) e^{ikt} dt.
\end{equation}

To demonstrate \eqref{ukelas} and for the further proofs we recall the well-known \cite{CH}, p. 711, integral representation of the solution
to \eqref{Cauchyelas}
$$
{\bf U}(x,k)=
{\bf V}(-c_s^2\Delta{\bf f}_1+
(c_p^2-c_s^2)\nabla div{\bf f}_1)(x,t)+
\partial_t
{\bf V}(-c_s^2\Delta{\bf f}_0+
(c_p^2-c_s^2)\nabla div{\bf f}_0)(x,t)+
$$
\begin{equation}
\label{integralr}
\partial^2_t
{\bf V}({\bf f}_1)(x,t)+
\partial^3_t
{\bf V}({\bf f}_0)(x,t),
\end{equation}
where
$$
{\bf V}({\bf F})(x,t)=
\frac{1}{c_p^2-c_s^2}
(\int_{|x-y|<c_st}\frac{c_p-c_s}{c_s}{\bf F}(y)dy+
\int_{c_st<|x-y|<c_pt}\frac{c_pt-|x-y|}{|x-y|}
{\bf F}(y)dy),
$$
which is valid at least when
${\bf f}_0\in C_0^3(\Omega),
{\bf f}_1\in C_0^3(\Omega)$.
\eqref{integralr} implies both the finite speed of the propagation and the Huygens' principle for ${\bf U}$, in particular,
\begin{equation}
\label{Huygenselas}
{\bf U}(x,t)=0\;\text{when}\;
c_pt+D<|x-a|\;\text{or when}\;
D<c_s t.
\end{equation}
 Approximating the initial data by smooth functions one can extend \eqref{Huygenselas} onto 
${\bf f}_0\in H^0(\Omega),
{\bf f}_1\in H^{-1}(\Omega)$.

To show \eqref{ukelas} let  
\begin{equation}
\label{uU}
\mathbf{u}^*(x,k)= \frac{1}{\sqrt{2\pi}}\int_{0}^{\infty} \mathbf{U}(\mathbf{x},t)e^{ikt} dt, \;\text{when}\; k>0.
\end{equation}
Using \eqref{Huygenselas} we conclude that 
\begin{equation}
\label{uUcurldiv}
{\bf curl}{\bf u}^*(x,k)= \frac{1}{\sqrt{2\pi}}\int_{0}^{\infty} {\bf curl}{\bf U}(x,t)e^{ikt} dt, \;
div\mathbf{u}^*(x,k)= \frac{1}{\sqrt{2\pi}}\int_{0}^{\infty} div {\bf U}(x,t)e^{ikt} dt\;
\text{when}\; k>0.
\end{equation}
Applying ${\bf curl}$ and $div$ to  \eqref{Helas} we  will have  
\begin{equation}
\label{Hcurl}
\quad  \mu \Delta {\bf curl}{\bf u}+ k^2 \rho {\bf curl}{\bf u}= -{\bf curl}({\bf f}_1+ik{\bf f}_1) ,\;
(\lambda+2\mu) \Delta div {\bf u}+ k^2 \rho  div {\bf u} = - div ({\bf f}_1+ik{\bf f}_0)\;\text{on}\;
\mathbb{R}^3.   
\end{equation}
In addition, from the integral representation of solutions to the Helmholtz equation  it follows that
\begin{equation}
\label{RadiationCondition1}
\lim_{r\rightarrow \infty} r(\partial _r {\bf curl}{\bf u}-\mathrm{i}k_p {\bf curl}{\bf u})=0, \quad \lim_{r\rightarrow \infty} r(\partial _r div {\bf u} -\mathrm{i}k_s div {\bf u})=0, \quad r=|\mathbf{x}|,
\end{equation}
provided $k>0$. Applying ${\bf curl}$ and $div$ to  \eqref{Cauchyelas} we  will have
\begin{equation}
\label{Cauchyelascurl}
\rho\partial_t^2 {\bf curl}{\bf U} -\mu\Delta {\bf curl} {\bf U}=0 \; \;\text{on}\;\mathbb{R}^3\times(0,+\infty),\;
{\bf curl}{\bf U}( ,0)={\bf curl} {\bf f}_0,\;\partial_t {\bf curl}{\bf U}( ,0)=-{\bf curl}{\bf f}_1\;
\text{on}\;\mathbb{R}^3,
\end{equation}
\begin{equation}
\label{Cauchyelasdiv}
\rho\partial_t^2  div{\bf U} - div{\bf curl} {\bf U}=0 \;\text{on}\;\mathbb{R}^3\times(0,+\infty),\;
div {\bf U}( ,0)= div {\bf f}_0,\;\partial_t div {\bf U}( ,0)=- div {\bf f}_1\;
\text{on}\;\mathbb{R}^3.
\end{equation}
As shown in \cite{CIL2016}, \eqref{Hcurl}, \eqref{RadiationCondition1}, {\eqref{Cauchyelascurl}, \eqref{Cauchyelasdiv} imply that
\begin{equation*}
{\bf curl}{\bf u}(x,k)=\frac{1}{\sqrt{2\pi}} \int_{0}^{\infty} {\bf curl}{\bf U}(x,t)e^{ikt} dt, \;
div\mathbf{u}(x,k)= \frac{1}{\sqrt{2\pi}}\int_{0}^{\infty} div {\bf U}(x,t)e^{ikt} dt\;
\text{when}\; k>0.
\end{equation*}
Therefore from \eqref{uUcurldiv} we have
\begin{equation*}
\nabla {\bf curl}(\mathbf{u}-{\mathbf{u}}^*)(x,k)=\mathbf{0},\; div (\mathbf{u}-{\mathbf{u}}^*)(x,k)=\mathbf{0} \quad \textrm{on} \quad  \mathbb{R}^3
\end{equation*}
and hence 
\begin{equation}
\label{harm}
\Delta (\mathbf{u}-{\mathbf{u}}^*)(\mathbf{x},k)=\mathbf{0} \quad \textrm{on} \quad  \mathbb{R}^3\;\text{when}\;0<k.
\end{equation}
Due to integral representations \eqref{int}, \eqref{uU} and the Huygense' principle \eqref{Huygenselas}
the functions ${\bf u}(x,k), 
{\bf u}^*(x,k)$ are entire analytic with respect to $k$
therefore \eqref{harm} holds for all complex $k$. If $0<k_2$ 
From the integral representation \eqref{intelas} function $\mathbf{u}(x,k)$ decays exponentially
as $|x|$ goes to $+\infty$. Using \eqref{uU}, \eqref{Huygenselas} as in \cite{EI}, section 4, one can show that
${\bf u}^*$ also decays exponentially in $|x|$ at fixed $k, 0<k_2$. By the Liouville's Theorem \eqref{harm} and exponential decay imply that 
${\bf u}(x,k)={\bf u}^*(x,k)$ when $0<k_2$. Using analyticity with respect to $k$
 and passing to the limit (as $k_2$ goes to zero) we obtain this equality and therefore \eqref{ukelas} for any real $k$.

\begin{lemma}
\label{Lemma32}
Let function $\mathbf{u}$ be a solution to the forward problem \eqref{Helas}-\eqref{radiationelas} with $\mathbf{f}_1 \in H^3 (\Omega)$ and $\mathbf{f}_0 \in H^4 (\Omega)$, $ supp\mathbf{f}_1, supp\mathbf{f}_0 \subset \Omega $. Then 
\begin{equation}
\label{boundu0}
\int_{k} ^{\infty}  \parallel \mathbf{u}(,\omega) \parallel^2 _{(0)} (\partial \Omega) d\omega \leq C  k^{-2}\Big (  \parallel \mathbf{f}_1 \parallel_{(1)}^{2}(\Omega)+\parallel \mathbf{f}_0 \parallel_{(2)}^{2}(\Omega)  \Big)
\end{equation}
and
\begin{equation}
\label{boundu1}
\int_{k} ^{\infty}(\omega^2    \parallel \mathbf{u}(,\omega) \parallel^2 _{(0)} (\partial \Omega)+ \parallel \mathbf{u}(,\omega) \parallel^2 _{(1)} (\partial \Omega))d\omega \leq C  k^{-2}(  \parallel \mathbf{f}_1 \parallel_{(2)}^{2}(\Omega)+\parallel \mathbf{f}_0 \parallel_{(3)}^{2}(\Omega)  \Big).
\end{equation} 
\end{lemma}
\begin{proof}
Let $B$ be a ball of the radius $(c_pc_s^{-1}+1)D$ centred at a point $a\in \Omega$. Since the (maximal) speed of the propagation for the  dynamical elasticity system
is $c_s$, ${\bf U}(x,t)=0$
when $c_pt< |x-a|-D$, in particular, ${\bf U}=0$ on
$\partial B\times (0, c_s^{-1}D)$. Moreover, due to the Huygens' principle,
${\bf U}=0$ on $\Omega\times
(c_s^{-1}D,+\infty)$.
The standard energy estimate for the initial value problem
\eqref{Cauchyelas} in $B\times (0, c_s^{-1}D)$ with ${\bf U}=0$ on $\partial B\times (0, c_s^{-1}D)$ gives
\begin{equation*}
  \parallel \partial_t\mathbf{U}(,t) \parallel^{2}_{(0)}(\Omega) +
    \parallel \mathbf{U}(,t) \parallel_{(1)}(\Omega)\leq
    C(\|{\bf f}_0\|_{(1)}(\Omega)+
    \|{\bf f}_1\|_{(0)}(\Omega)), 0<t<c_s^{-1}D.
\end{equation*}
Observe that $\partial_t {\bf U}$ solves the  elasticity system \eqref{Cauchyelas} and satisfies the initial conditions
$$
\partial_t {\bf U}={\bf f}_1,\;
\partial^2_t {\bf U}=
\rho^{-1}(\mu\Delta {\bf f}_0+
(\lambda+\mu)\nabla div {\bf f}_0)\;\text{on}\;{\mathbb R}^3\times\{0\}.
$$
 Applying the same energy estimate we yield
\begin{equation*}
  \parallel \partial^2_t\mathbf{U}(,t) \parallel^{2}_{(0)}(\Omega) +
    \parallel \partial_t \mathbf{U}(,t) \parallel_{(1)}(\Omega)\leq
    C(\|{\bf f}_0\|_{(2)}(\Omega)+
    \|{\bf f}_1\|_{(1)}(\Omega)), 0<t<c_s^{-1}D.
\end{equation*}
By Trace Theorems for Sobolev spaces we obtain
\begin{equation}
\label{trace}
\| \partial_t {\bf U}( ,t)\|_{(0)}(\partial\Omega)\leq C
\| \partial_t {\bf U}( ,t)\|_{(1)}(\Omega)\leq
C(\|{\bf f}_0\|_{(2)}(\Omega)+
    \|{\bf f}_1\|_{(1)}(\Omega)), 0<t<c_s^{-1}D.
\end{equation}

Now,
\begin{equation*}
 \int_{k}^{\infty}   \parallel \mathbf{u}(,\omega) \parallel^{2}_{(0)}(\partial \Omega) d\omega\leq 
k^{-2} \int_{k}^{\infty}   \parallel \omega^2\mathbf{u}(,\omega) \parallel^{2}_{(0)}(\partial \Omega) d\omega\leq
\end{equation*}
\begin{equation*} 
k^{-2} \int_{0}^{\infty}   \parallel \partial_t \mathbf{U}(,t)^2 \parallel^{2}_{(0)}(\partial \Omega) d t=
k^{-2} \int_{0}^{c_s^{-1}D}   \parallel \partial_t \mathbf{U}(,t)^2 \parallel^{2}_{(0)}(\partial \Omega) d t
\leq
C(\|{\bf f}_0\|^2_{(2)}(\Omega)+
    \|{\bf f}_1\|^2_{(1)}(\Omega)).
\end{equation*}
due to the Huygens' principle and \eqref{trace}.

This completes a proof of \eqref{boundu0}.

Now we will demonstrate \eqref{boundu1}.

As above,  $\partial^2_t {\bf U}$ solves the  elasticity system \eqref{Cauchyelas} and satisfies the initial conditions
$$
\partial_t^2 {\bf U}=
\rho^{-1}(\mu\Delta {\bf f}_0+
(\lambda+\mu)\nabla div {\bf f}_0),\;
\partial^3_t {\bf U}=
\rho^{-1}(\mu\Delta {\bf f}_1+
(\lambda+\mu)\nabla div {\bf f}_1)\;\text{on}\;{\mathbb R}^3\times\{0\},
$$
so
\begin{equation*}
  \parallel \partial^2_t\mathbf{U}(,t) \parallel^{2}_{(1)}(\Omega) \leq
    C(\|{\bf f}_0\|^2_{(3)}(\Omega)+
    \|{\bf f}_1\|^2_{(2)}(\Omega)),
\end{equation*}
and by Trace Theorems 
\begin{equation}
\label{traceU}
\| \partial_t^2 {\bf U}( ,t)\|_{(0)}(\partial\Omega)\leq
C(\|{\bf f}_0\|_{(3)}(\Omega)+
    \|{\bf f}_1\|_{(2)}(\Omega)), 0<t<c_s^{-1}D.
\end{equation}

Hence,
\begin{equation*}
 \int_{k}^{\infty}  \omega^2 \parallel \mathbf{u}(,\omega) \parallel^{2}_{(0)}(\partial \Omega) d\omega\leq 
k^{-2} \int_{k}^{\infty}  \omega^4 \parallel \mathbf{u}(,\omega) \parallel^{2}_{(0)}(\partial \Omega) d\omega\leq
\end{equation*}
\begin{equation*} 
k^{-2} \int_{0}^{\infty}   \parallel \partial_t^2 \mathbf{U}(,t) \parallel^{2}_{(0)}(\partial \Omega) d t=
k^{-2} \int_{0}^{c_s^{-1}D}   \parallel \partial_t^2 \mathbf{U}(,t) \parallel^{2}_{(0)}(\partial \Omega) d t
\leq
C(\|{\bf f}_0\|^2_{(3)}(\Omega)+
    \|{\bf f}_1\|^2_{(2)}(\Omega)),
\end{equation*}
due to \eqref{traceU}.

To complete the proof of Lemma 3.2 we apply the proof of \eqref{boundu0} to $\partial_j {\bf U}, j=1,2,3$ instead of ${\bf U}$. 

 The proof is complete.
\end{proof}

As above the next result follows almost immediately from the Huygens' principle for the initial value problem and the known bounds for  initial boundary value hyperbolic problems.

\begin{lemma}
\label{Lemma33}
Let ${\bf U}$ be a solution to
\eqref{Cauchyelas} with 
${\bf f}_1 \in L^2(\Omega)$,
 ${\bf f}_0\in H^1(\Omega)$, $ supp {\bf f}_0, supp {\bf f}_1\subset \Omega$.
Then there is $C$ such that
\begin{equation}
\label{exact0e}
\|{\bf f}_0\|^2_{(0)}(\Omega)+
\|{\bf f}_1\|^2_{(-1)}(\Omega) \leq
C\| {\bf U} \|^2_{(0)}(\partial\Omega\times
(0,c_s^{-1}D))
\end{equation}
and
\begin{equation}
\label{exacte}
\|{\bf f}_0\|^2_{(1)}(\Omega)+
\|{\bf f}_1\|^2_{(0)}(\Omega) \leq
C\| {\bf U} \|^2_{(1)}(\partial\Omega\times(0,c_s^{-1}D)).
\end{equation}
\end{lemma}

\begin{proof}

Since $supp {\bf f}_0, supp {\bf f}_1 \subset \Omega$ from the Huygens' principle it follows that ${\bf U}=0$ on $\Omega\times (c_s^{-1}D,+\infty)$. Now \eqref{exact0e}, \eqref{exacte} follow from the generalizations \cite{BL2002}, Theorem 1, \cite{GZ},Theorem 1.1, Lemma 3.5, of Sakamoto energy estimates \cite{S} for the initial boundary value problem applied to the initial boundary value problem
$$
\rho \partial_t^2 {\bf U} -\mu\Delta {\bf U}-(\lambda+\mu) \nabla div {\bf U}=0 \;\text{on}\;\Omega\times(0,c_s^{-1}D),\;
{\bf U}( ,c_s^{-1}D)=0,\;\partial_t {\bf U}( ,c_s^{-1}D)=0\;
\text{on}\;\Omega.
$$
Observe that the results in \cite{GZ} claim that the solution is contained in the corresponding function spaces. Since the operator mapping the initial data into the lateral boundary data is closed in these spaces, the  bounds
\eqref{exact0e}, \eqref{exacte} follow from the Closed Graph Theorem.

\end{proof}

Finally, we are ready to prove the increasing stability estimate of Theorem 1.2.

\begin{proof}

We start with a proof of \eqref{stability0elas}.

Without loss of generality, we can assume that $\varepsilon <1$ and
$2\pi(D+1)E^{-\frac{1}{4}} <1$, otherwise the bound
\eqref{stability0elas} is straightforward.

In \eqref{int0elas} we let
\begin{equation}
\label{Kelas}
k=K^{\frac{2}{3}} E^{\frac{1}{4}},\;
\text{when}\; 2^{\frac{1}{4}}K^{\frac{1}{3}}<
E^{\frac{1}{4}},\;\text{and}\;k=K,\;
\text{when}\;E^{\frac{1}{4}}\leq
2^{\frac{1}{4}}K^{\frac{1}{3}}.
\end{equation}
If $2^{\frac{1}{4}}K^{\frac{1}{3}}<
E^{\frac{1}{4}}$, then from \eqref{I0omegaelas},
\eqref{boundharm2}, and \eqref{Kelas} we obtain
\begin{align*}
|I_{0,e}
(k)| & \leq
e^{2(D+1)k} e^{-\frac{2}{\pi}
\left(\left(\frac{k}{K}\right)^4 -1\right)^{-\frac{1}{2}}E} CM_{2,e}^2 \\
&\leq
CM_{2,e}^2 e^{2(D+1)K^{\frac{2}{3}}E^{\frac{1}{4}}-
 \frac{2}{\pi}\left(\frac{K}{k}\right)^2 E} =
 CM_{2,e}^2 e^{-2K^{\frac{2}{3}} \frac{1}{\pi}E^{\frac{1}{2}}\left(1-\pi(D+1)E^{-\frac{1}{4}}\right)}.
\end{align*}
Using the assumption at the beginning of the proof and the elementary inequality $e^{-y}\leq \frac{6}{y^3}$  when $0<y$, we conclude that
\begin{equation}
\label{I0helas}
|I_{0,e}(k)|\leq
CM_{2,e}^2 \frac{1}{K^2 E^{\frac{3}{2}}\left(1-\pi(D+1)E^{-\frac{1}{4}}\right)^3}.
\end{equation}
On the other hand, if $E^{\frac{1}{4}} \leq 2^{\frac{1}{4}}K^{\frac{1}{3}}$, then $k=K$ and  from (\ref{Kelas}) we derive that
\begin{equation}
\label{I0lelas}
|I_{0,e}(k)|\leq 2\varepsilon^2.
\end{equation}

Hence, using \eqref{int0elas}, \eqref{I0helas}, \eqref{I0lelas}, and (\ref{Kelas}) we yield
\begin{align}
\int_{\partial\Omega}\int_{-\infty}^{+\infty}
|{\bf u}(x,\omega)|^2 d\omega d\Gamma(x)
& =
I_{0,e}(k)+ \int_{\partial\Omega}\int_{k<|\omega|}
|{\bf u}((x,\omega)|^2 d\omega d\Gamma(x) \nonumber \\
& \leq \varepsilon^2+ CM_{2,e}^2\frac{1}{K^2 E^{\frac{3}{2}}}+
C\frac{\|{\bf f}_0\|_{(2)}^2+\|{\bf f}_1\|_{(1)}^2}
{1+K^{\frac{4}{3}}E^{\frac{1}{2}}} , \label{data1elas}
\end{align}
where we used (\ref{boundu0}) and (\ref{Kelas}).

By Lemma \ref{Lemma33}, we finally derive
$$
\|{\bf f}_0\|^2_{(0)}(\Omega)+
\|{\bf f}_1\|^2_{(-1)}(\Omega)  \leq
C\|{\bf U}\|^2_{(0)}(\partial\Omega\times (0, c_s^{-1}D))  \leq
C\|{\bf U}\|^2_{(0)}(\partial\Omega\times {\mathbb R})= 
$$
$$
C\int_{\partial\Omega}
\int_{-\infty}^{+\infty}
 |{\bf u}(x,\omega)|^2 d\omega d\Gamma(x)  \leq
C \left(\varepsilon^2+ M_{2,e}^2\frac{1}{K^2 E^{\frac{3}{2}}}+
\frac{\|{\bf f}_0\|_{(2)}^2+\|{\bf f}_1\|_{(1)}^2}
{1+K^{\frac{4}{3}}E^{\frac{1}{2}}}\right)
$$
due to the Parseval's identity and (\ref{data1elas}). Since
$$K^{\frac{4}{3}}E^{\frac{1}{2}}< K^2  E^{\frac{3}{2}},
$$
when $1<K, 1<E$, the proof of
\eqref{stability0elas} is complete.

\eqref{stabilityelas} can be derived similarly.
As above,
\begin{align}
\int_{-\infty}^{+\infty}
\omega^2 \| {\bf u}(,\omega)\|_{(0)}^2(\partial\Omega) d\omega  & =
I_{1,e}(k)+ 
\int_{k<|\omega|}
\| {\bf u}((,\omega)\|^2_{(0)}(\partial\Omega) d\omega  \nonumber \\
& \leq \varepsilon^2 + C M_3^2\frac{1}{K^2 E_e^{\frac{3}{2}}}+
C\frac{\|{\bf f}_0\|_{(3)}^2+\|{\bf f}_1\|_{(2)}^2}
{1+K^{\frac{4}{3}}E_e^{\frac{1}{2}}}. \label{data2elas}
\end{align}
and
\begin{align}
\int_{-\infty}^{+\infty}
\| {\bf u}(,\omega)\|_{(1)}(\partial\Omega)^2 d\omega  & =
I_{2,e}(k)+ 
\int_{k<|\omega|}
\| {\bf u}((,\omega)\|_{(1)}(\partial\Omega)^2 d\omega  \nonumber \\
& \leq \varepsilon^2 + C M_3^2\frac{1}{K^2 E_e^{\frac{3}{2}}}+
C\frac{\|{\bf f}_0\|_{(3)}^2+\|{\bf f}_1\|_{(2)}^2}
{1+K^{\frac{4}{3}}E_e^{\frac{1}{2}}}. \label{data3elas}
\end{align}

By using \eqref{exacte}, we finally derive
\begin{align*}
\|{\bf f}_0\|^2_{(1)}(\Omega)+
\|{\bf f}_1\|^2_{(0)}(\Omega)) &  \leq
C(\|\partial_t U\|^2_{(0)}(\partial\Omega\times {\mathbb R})+
\|\nabla_{\tau} U \|^2_{(0)}(\partial\Omega\times {\mathbb R})) \\
& =
C\left(
\int_{-\infty}^{+\infty}
\omega^2 \|{\bf u}(,\omega)\|_{(0)}^2(\partial\Omega) d\omega + \int_{-\infty}^{+\infty}
 \|{\bf u}( ,\omega)\|_{(1)}^2(\partial\Omega) d\omega  \right) \\
& \leq
C \left(\varepsilon_e^2+ M_3^2\frac{1}{K^2 E_e^{\frac{3}{2}}}+
\frac{\|{\bf f}_0\|_{(3)}^2+\|{\bf f}_1\|_{(2)}^2}
{1+K^{\frac{4}{3}}E_e^{\frac{1}{2}}}\right)
\end{align*}
due to the Parseval's identity and \eqref{data2elas},  \eqref{data3elas}. Since
$$K^{\frac{4}{3}}E_e^{\frac{1}{2}}< K^2  E_e^{\frac{3}{2}},
$$
when $1<K, 1<E_e$, the proof of is complete.
\end{proof}

\section{Conclusion}

The next analytical issue is to obtain explicit constants
$C$ in the stability estimates of Theorem \ref{thm1} for some simple but important domains $\Omega$, like a sphere or a cube. This seems to be quite realistic. Another possible development is to get these estimates when $(0,K)$ is replaced
by $(K_*,K)$ with, say, $K_*=K/2$.
One expects the stability results to be extended onto more general scalar elliptic operators and elasticity systems satisfying non trapping (pseudo convexity) conditions in $\Omega$. One does not have the Huygens' principle in the general case, but 
the exact controllability theory for corresponding hyperbolic equations is developed in  \cite{T} and there are certain results for the elasticity system, obtained by the multipliers  method.
The needed scattering theory also available, although not so transparent and explicit as for the Helmholtz equation and the classical elasticity system.
In this more general case it is difficult to expect constants $C$ to be explicit.
By using the sharp uniqueness of the continuation results for the isotropic dynamical Maxwell and elasticity systems \cite{EINT} one expects to obtain uniqueness in the inverse source problems for these systems. Applying available exact observability for these systems and the method of this paper one anticipates increasing stability in the inverse source problems. 
It is feasible and interesting to obtain similar results in the case of the data on a part of $\partial \Omega$. Some parts of the methods of \cite{CIL2016} and of this paper can be useful, although in the place of nearly Lipschitz stability bounds \eqref{stability10},
\eqref{stability1}, \eqref{stability0elas}, \eqref{stabilityelas} one probably should expect near H\"older type stability originated from the stability estimates in the lateral Cauchy problem for hyperbolic equations \cite{I} section 3.4. It is important to collect further numerical evidence of the increasing stability for more complicated geometries and sources. Some very convincing numerical results are given in \cite{CIL2016},
\cite{IL}.

One can look at the different inverse source problem which is the linearized inverse problem for the Schr\"odinger potential: find $f$ (supported in $\Omega\subset \mathbb R^3$) from
\begin{align*}
\int_{\Omega} f(y) \frac{e^{ki|x-y|}}{|x-y|}\frac{e^{ki|z-y|}}{|z-y|}dy
\end{align*}
given for $x,z\in \Gamma \subset \partial \Omega$, where one expects quite explicit increasing stability bounds for $f$ for a large $k$. 
Analytic study of this problem is started in \cite{I3}, \cite{ILW2016}.

\section*{Aknowledgement}
This research is supported in part by the Emylou Keith and Betty Dutcher Distinguished Professorship 
and the NSF grant DMS 15-14886. 

\bibliographystyle{amsplain}
\providecommand{\bysame}{\leavevmode\hbox to3em{\hrulefill}\thinspace}
\providecommand{\MR}{\relax\ifhmode\unskip\space\fi MR }
\providecommand{\MRhref}[2]{%
  \href{http://www.ams.org/mathscinet-getitem?mr=#1}{#2}
}
\providecommand{\href}[2]{#2}

\end{document}